\documentclass{article}

\usepackage{booktabs, multirow, bm}
\usepackage{amsmath, amssymb, amsthm, mathrsfs}
\usepackage{indentfirst}
\usepackage{graphicx, tikz, subfigure}
\usepackage{geometry}
\usepackage{longtable}
\usepackage{caption}
\usepackage{multirow}
\numberwithin{equation}{section}
\usepackage{xcolor}
\usepackage{ulem}
\usepackage[modulo]{lineno} 
\usepackage{hyperref}
\allowdisplaybreaks[4]
\usepackage{cleveref}
\crefname{section}{Section}{Sections}
\crefname{figure}{Figure}{Figures}
\crefname{table}{Table}{Tables}
\crefname{equation}{}{}
\crefname{theorem}{Theorem}{Theorems}
\crefname{lemma}{Lemma}{Lemmas}
\crefname{remark}{Remark}{Remarks}
\crefname{problem}{Problem}{Subproblems}

\newtheorem{theorem}{Theorem}[section]

\newtheorem{remark}{Remark}[section]
\newtheorem{lemma}{Lemma}[section]

\theoremstyle{definition}

\crefname{ip}{Co-inversion Problem}{ips}
\newtheoremstyle{MyThmStyle}
{}
{}
{}
{}
{\bfseries}
{}
{ }
{\thmname{#1\thmnumber{ #2\hspace{0.5em}}}\thmnote{(#3)}}
\theoremstyle{MyThmStyle}

\crefname{subisp}{inverse source problem}{ips}
\crefname{subiop}{inverse obstacle problem}{ips}
\graphicspath{{figurenew/}}
\definecolor{bananamania}{rgb}{0.98, 0.91, 0.71}

\begin{document}
	
	\title{Stochastic positivity-preserving symplectic splitting methods for stochastic Lotka--Volterra predator-prey model}

	\author{
		Liying Zhang\thanks{School of Mathematical Science, China University of Mining and Technology, Beijing 100083, China, lyzhang@lsec.cc.ac.cn }, Xinyue Kang\thanks{School of Mathematical Science, China University of Mining and Technology, Beijing 100083, China, SQT2300702045@student.cumtb.edu.cn },
	Lihai Ji\thanks{
		Institute of Applied Physics and Computational Mathematics, Beijing 100094, China, jilihai@lsec.cc.ac.cn} \thanks{Shanghai Zhangjiang Institute of Mathematics, Shanghai 201203, China, jilihai@lsec.cc.ac.cn}}
	
	\date{}
	\maketitle
\begin{abstract}
	In this paper, we present two stochastic positive-preserving symplectic methods for the stochastic Lotka--Volterra predator-prey model driven by a multiplicative noise. To inherit the intrinsic characteristic of the original system, the stochastic Lie--Trotter splitting method and the stochastic Strang splitting method are introduced, which are proved to preserve the positivity of the numerical solution and possess the discrete stochastic symplectic conservation law as well. By deriving the uniform boundedness of the $p$-th moment of the numerical solution, we prove that the strong convergence orders of these two methods are both one in the $L^2(\Omega)$-norm. Finally, we validate the theoretical results through two and four dimensional numerical examples.
\end{abstract}
	

\section{Introduction}
\label{Introduction}
The Lotka--Volterra (LV) predator-prey model, proposed by Lotka and Volterra, serves as a foundational framework in ecology for describing the dynamic interactions between prey and predator populations \cite{A1,V1}. It characterizes the temporal fluctuations in population sizes of both species \cite{J1}. However, in real-world ecosystems, species survival is frequently subjected to stochastic environmental perturbations \cite{G1,Q1}, which directly influence the survival and reproduction rates of prey and predators, thereby inducing population fluctuations \cite{L1}. To address this issue, mathematicians have integrated environmental white noise into ecological models to simulate stochastic factors’ impacts. For example,  \cite{A2} showed that incorporating randomness captures phenomena like abrupt population collapses or outbreaks — outcomes unanticipated by the deterministic LV framework.

The main objective of this paper is to numerically investigate the following $2d$-dimensional stochastic LV predator-prey model driven a multiplicative noise
\begin{equation}\label{eq:1}
	\begin{cases}
		d X(t)=\overline{X}(t)\left[\left(-\Gamma^{(2)} Y(t)+\eta^{(2)}\right) d t+\Sigma^{(2)} d W(t)\right] ,\\[2.5mm]
		d Y(t)=\overline{Y}(t)\left[\left(\Gamma^{(1)} X(t)-\eta^{(1)}\right) d t+\Sigma^{(1)} dW(t)\right],\\[2.5mm]
		X(0)=X_0,~Y(0)=Y_0,
	\end{cases}
\end{equation}
where $X(t) = \left(x_{1}(t), \cdots, x_{d}(t)\right)^{\top}$ and $Y(t) = \left(y_{1}(t), \cdots, y_{d}(t)\right)^{\top}$ denote the population densities of $d$ prey species and $d$ predator species at time $t$, respectively. $\overline{X} = \text{diag}\left\{x_{1}, \cdots, x_{d}\right\} \in \mathbb{R}^{d \times d}$ represents a diagonal matrix and diagonal elements are the components of the vector $X$. Similarly, $\overline{Y} = \text{diag}\{y_1, \cdots, y_d\}.$ For any vector or matrix $A$ , by using the notation $A>0$, we mean all the entries in $A$ are positive. The matrix $\Gamma^{(1)} = [\gamma_{ij}^{(1)}]_{d \times d} > 0$ represents the conversion rate from prey consumption to predator reproduction, indicating the proportion of new predator individuals generated for each unit of prey consumed in the ecosystem. The matrix $\Gamma^{(2)} = [\gamma_{ij}^{(2)}]_{d \times d} > 0$ denotes the mortality rate of prey owing to predation events. The vector $\eta^{(1)} = (\eta_{1}^{(1)}, \cdots, \eta_{d}^{(1)})^{\top} > 0$ represents the natural mortality rate of the $d$ predator species in the absence of food, whereas the vector $\eta^{(2)} = (\eta_{1}^{(2)}, \cdots, \eta_{d}^{(2)})^{\top} > 0$ represents the natural growth rate of the $d$ prey species in the absence of predators. Additionally, $W = (W_1, \cdots, W_m)^{\top}$ is an $m$-dimensional standard Wiener process, where $W_i,~i=1, \cdots, m$ are $m$ independent one-dimensional Wiener processes defined on a probability space $(\Omega, ~\mathcal{F},~\mathbb{P})$. The matrices $\Sigma^{(k)} = [\delta_{ij}^{(k)}]_{d \times m},~ k=1,2$ quantify the magnitude of the noise.  In recent years, significant progress has been made in theoretical studies of the stochastic LV model. For example, \cite{Hong} proved that (\ref{eq:1}) admits a unique positive solution, the $p$-th moment of the solution remains uniformly bounded and the model preserves the stochastic symplectic structure. Although the stochastic LV model admits a unique bounded solution, its multi-dimension and super-linearity of both the drift and diffusion coefficients make it challenging to derive exact solutions.

Various approximation methods have been introduced in the literature for this problem. To preserve the positivity of the solution, \cite{X2} proposed explicit truncated EM method; \cite{Y1} extended the truncated EM method to the stochastic LV model with super-linearly growing coefficients; \cite{Y2} presented a positivity preserving Lamperti transformed EM method; \cite{F1} designed an efficient and convenient method via operational matrices. However, these positivity-preserving methods fail to preserve the stochastic symplectic structure for the stochastic LV model \eqref{eq:1}. Recently, \cite{Hong} proposed a class of stochastic Runge--Kutta methods for the stochastic LV model, which is proved to preserve positivity of the numerical solution and possess the discrete stochastic symplectic conservation law as well. To the best of our knowledge, the application of the splitting methods to study positivity-preserving symplectic numerical methods for the stochastic LV model remains unexplored in existing literature. 

This paper aims to construct positivity-preserving symplectic numerical methods for \eqref{eq:1} by utilizing the stochastic splitting technique. When the matrices $\Gamma^{(k)}, ~k=1,2$ are diagonal, we reformulate \eqref{eq:1} as a stochastic Hamiltonian system and apply the stochastic splitting technique to decouple it into two exactly solvable subsystems. By applying the stochastic Lie--Trotter and the stochastic Strang composition operations to the exact solutions of these subsystems, we derive the stochastic Lie--Trotter and the stochastic Strang splitting methods for \eqref{eq:1}, respectively. The resulting methods maintain an explicit structure, enabling straightforward verification of positivity preservation of the numerical solution. Furthermore, the stochastic splitting methods possess the stochastic symplecticity. Based on the fundamental theorem on the mean-square order of convergence \cite{M} and utilizing the uniform boundedness of the numerical solution, we can prove that these two stochastic positivity-preserving symplectic splitting methods converge with global order one in the $L^2(\Omega)$-norm. Finally, we validate the theoretical results through two and four dimensional numerical examples.

The structure of this paper is as follows: in section \ref{solution},  some preliminaries are collected and some properties of stochastic LV model, including regularity and stochastic symplecticity, are also considered. In section \ref{sec:positivity}, two stochastic positivity-preserving symplectic splitting methods are proposed and our main results are stated: in section \ref{sec:symplectic} we give some conditions to guarantee that two given stochastic splitting methods are symplectic; in section \ref{sec:bound} we prove the unique existence and regularity of the numerical solution of the stochastic splitting methods; section \ref{sec:convergence order} is devoted to the proof of the convergence order of the stochastic positivity-preserving symplectic splitting methods. Finally, two numerical experiments are performed in section \ref{sec:numerical} to validate the effectiveness of the proposed methods.

\section{Preliminaries}\label{solution}
In this section, we present some preliminaries for the analysis of the stochastic LV model \eqref{eq:1}. And some properties of the stochastic LV model, including regularity and stochastic symplecticity are also considered. We refer to \cite[Section 2]{Hong} and references therein for more details.

Throughout this paper, the constants C may be different from line to line. When it is necessary to indicate the dependence on some parameters, we will use the notation $C(\cdot)$. For instance, $C(X_{0}, Y_{0},T,p)$ is a constant depending on $X_{0},~Y_{0},~T$ and $p$. For any positive integer $n \in \mathbb{N}_{+} $ and vectors  $U=\left(u_{1}, \cdots, u_{n}\right)^{\top} \in \mathbb{R}^{n},~V=\left(v_{1}, \cdots, v_{n}\right)^{\top} \in \mathbb{R}^{n}$, we define $U {\ast} V:=\left(u_{1} v_{1}, \cdots, u_{n} v_{n}\right)^{\top}.$ 

The following lemma focuses on the global well-posedness and positivity of the solution to \eqref{eq:1}.
\begin{lemma}\label{p1}
	For any deterministic initial datum $(X_0^\top,~Y_0^\top)^{\top}\in\mathbb{R}_+^{2d}$, the stochastic LV model \eqref{eq:1} has a unique solution $(X^\top(t),~Y^\top(t))^{\top}$. Furthermore, for all $0\leq t\leq \infty$, it holds $(X^\top(t),~Y^\top(t))^{\top}\in\mathbb{R} _+ ^{2d}$.
\end{lemma}

By utilizing the positivity of the solution of the stochastic LV model \eqref{eq:1}, the uniform boundedness of the solution can be obtained. The proof is analogous to that of \cite[Proposition 2.1]{Hong} and thus is omitted.
\begin{lemma}\label{p2}
	For any $p\geq1$ and a deterministic initial value $(X_{0}^\top,~Y_{0}^\top)^{\top}\in\mathbb{R}_{+}^{2d}$, the solution of \eqref{eq:1} is uniformly bounded and satisfies $$\sup_{t\in[0,T]}\mathbb{E}\Big[\sum_{i=1}^{d}\left(p_{i}x_{i}(t)+q_{i}y_{i}(t)\right)\Big]^{p}\leq C,$$ 
	where the positive constant $C=C(X_0, Y_0,\Gamma^{(1)},\Gamma^{(2)},\eta^{(1)},\eta^{(2)},\Sigma^{(1)},\Sigma^{(2)})$ and $p_{i}$ and $q_{i}$ are positive constants given by \cite[Eqs. 2.2-2.4]{Hong}.
\end{lemma}

Specifically, when $\Gamma^{(k)}$ with $ k=1,2$ are diagonal matrices, \cite{Hong} showed that the stochastic LV model \eqref{eq:1} can be reformulate as a non-canonical stochastic Hamiltonian system in the Stratonovich sense, that is
\begin{equation}\label{eq:2}
	d Z(t)=K^{-1}(Z(t)) \nabla_{Z} H_{1}(Z(t)) d t+K^{-1}(Z(t)) \nabla_{Z} H_{2}(Z(t)) \circ d W(t),
\end{equation}
where $Z=(X^\top,~Y^\top)^\top$ with  $Z_{0}=\left(X_{0}^\top,~Y_{0}^\top\right)^\top\in   \mathbb{R}_{+}^{2 d}$. $K$ is an invertible skew-symmetric matrix which satisfies
\begin{equation}
	\begin{aligned}
		K(Z)= & {\left[\begin{array}{cc}
				0 & -K^{*}(Z) \\[2mm] 
				K^{*}(Z) & 0
			\end{array}\right], \quad K^{*}(Z)=\operatorname{diag}\left\{\frac{1}{x_{1} y_{1}}, \cdots,\frac{1}{x_{d} y_{d}}\right\}},\\
	\end{aligned}
\end{equation}
and the Hamiltonians $H_{1}(Z),~H_{2}(Z)$ are given by
\begin{equation}\label{hamiltos}
	\begin{split}
		H_{1}(Z)= & \sum_{i=1}^{d}\left[-\gamma_{i i}^{(1)} x_{i}+\left(\eta_{i}^{(1)}+\frac{1}{2} \sum_{j=1}^{d}\left(\sigma_{i j}^{(1)}\right)^{2} \ln x_{i}\right)\right] \\
		& +\sum_{i=1}^{d}\left[-\gamma_{i i}^{(2)} y_{i}+\left(\eta_{i}^{(2)}-\frac{1}{2} \sum_{j=1}^{d}\left(\sigma_{i j}^{(2)}\right)^{2} \ln y_{i}\right)\right], \\
		H_{2}(Z)= & {\left[\sum_{i=1}^{d}\left(-\sigma_{i 1}^{(1)} \ln x_{i}+\sigma_{i 1}^{(2)} \ln y_{i}\right), \cdots, \sum_{i=1}^{d}\left(-\sigma_{i m}^{(1)} \ln x_{i}+\sigma_{i m}^{(2)} \ln y_{i}\right)\right]_{1 \times m} . }
	\end{split}
\end{equation}

Thereby, the stochastic LV model \eqref{eq:1} possesses the following stochastic symplectic geometric structure.
\begin{lemma}\label{p3}
	In the case of diagonal matrices $\Gamma^{(k)} $ with $k=1,2$, the phase flow of \eqref{eq:1} preserves the stochastic symplectic conservation law
	$$\mathrm{d} Z(t) \wedge K(Z(t)) \mathrm{d} Z(t)=\mathrm{d} Z_{0} \wedge K\left(Z_{0}\right) \mathrm{d} Z_{0}, \quad a.s.$$
	for all $t\geq0$. Equivalently, the phase flow $\varphi_{t}: Z_{0} \mapsto Z(t) $ satisfies
	\begin{align}\label{eq:15}
		\left[\frac{\partial\varphi _t(Z_0)}{\partial Z_0}\right]^\top K(\varphi _t(Z_0))\left[\frac{\partial\varphi _t(Z_0)}{\partial Z_0}\right]=K(Z_0).
	\end{align}
\end{lemma}
\section{Stochastic positivity-preserving symplectic splitting methods}\label{sec:positivity}
In this section, we will investigate the stochastic splitting methods for the stochastic LV model \eqref{eq:1} which are derived by the Lie--Trotter and the Strang splitting techniques. Moreover, we show the positivity, the symplecticity and the convergence order of the stochastic splitting methods.

Notice that each Hamiltonians $H_1(Z)$ and $H_2(Z)$ can be decomposed into two components
\begin{equation}\label{eq:42}
	H_1(Z) = H_{1,X}(X) + H_{1,Y}(Y),{\quad} H_2(Z) = H_{2,X}(X) + H_{2,Y}(Y),
\end{equation} 
respectively, where $H_{i,X}$ depends solely on $X(t)$ and $H_{i,Y}$ on $Y(t)$ for $i=1,2$. It follows from \eqref{hamiltos} that
\begin{align}{\notag} 
	H_{1,X}(X)= & \sum_{i=1}^{d}\left[-\gamma_{i i}^{(1)} x_{i}+\left(\eta_{i}^{(1)}+\frac{1}{2} \sum_{j=1}^{d}\left(\sigma_{i j}^{(1)}\right)^{2} \ln x_{i}\right)\right] ,\\{\notag} 
	H_{1,Y}(Y)= & \sum_{i=1}^{d}\left[-\gamma_{i i}^{(2)} y_{i}+\left(\eta_{i}^{(2)}-\frac{1}{2} \sum_{j=1}^{d}\left(\sigma_{i j}^{(2)}\right)^{2} \ln y_{i}\right)\right], \\{\notag}
	H_{2,X}(X)= & {\left[\sum_{i=1}^{d}\left(-\sigma_{i 1}^{(1)} \ln x_{i}\right), \cdots, \sum_{i=1}^{d}\left(-\sigma_{i m}^{(1)} \ln x_{i}\right)\right]_{1 \times m} ,}\\{\notag} 
	H_{2,Y}(Y)= &{\left[\sum_{i=1}^{d}\left(\sigma_{i 1}^{(2)} \ln y_{i}\right), \cdots, \sum_{i=1}^{d}\left(\sigma_{i m}^{(2)} \ln y_{i}\right)\right]_{1 \times m} ,{\notag}  }
\end{align}
which by the stochastic Hamiltonian system \eqref{eq:2}, implies that
\begin{equation}\label{eq:3}
	d Z(t)=K^{-1}(Z(t)) \nabla_{Z} H_{1,X}(X(t)) d t+K^{-1}(Z(t)) \nabla_{Z} H_{2,X}(X(t)) \circ d W(t),
\end{equation}
and
\begin{equation}\label{eq:4}
	d Z(t)=K^{-1}(Z(t)) \nabla_{Z} H_{1,Y}(Y(t)) d t+K^{-1}(Z(t)) \nabla_{Z} H_{2,Y}(Y(t)) \circ d W(t).
\end{equation}

By the It$\hat{\text{o}}$ formula, we can explicitly derive the exact solutions for the above two subsystems. For subsystem \eqref{eq:3}, the solution is given by
\begin{equation}\label{eq1111}
	\begin{array}{l}
		{\varphi} _{t}^{1}:\left\{\begin{array}{l}
			X(t)=X_{0}, \\[2mm]
			Y(t)=Y_{0} {\ast}  \exp  \left\{\left(\Gamma ^{(1)} X_{0}-\eta^{(1)}-\frac{1}{2} \Lambda^{(1)}\right) t+\Sigma^{(1)} W(t)\right\}.
		\end{array}\right. 
	\end{array}
\end{equation}
And the solution of subsystem \eqref{eq:4} satisfies
\begin{equation}\label{eq2222}
	\begin{array}{l}
		{\varphi} _{t}^{2}:\left\{\begin{array}{l}
			X(t)=X_{0} {\ast}  \exp  \left\{\left(-\Gamma ^{(2)} Y_{0}+\eta^{(2)}-\frac{1}{2} \Lambda^{(2)}\right) t+\Sigma^{(2)} W(t)\right\} ,\\ [2mm]
			Y(t)=Y_{0},
		\end{array}\right. 
	\end{array}
\end{equation}
where $\Lambda^{(k)}:=\left(\sum_{j=1}^{m}\left(\sigma_{1 j}^{(k)}\right)^{2}, \cdots, \sum_{j=1}^{m}\left(\sigma_{d j}^{(k)}\right)^{2}\right)^{\top},$ $k=1,2$. 

For the time interval $[0,T]$, we define the step size $h=T/N$, which partitions the domain uniformly into $N+1$ temporal nodes $0 = t_0 < t_1 < \cdots < t_N = T$. The numerical solution $Z_{n}$ approximates the exact solution $Z(t_{n})$  at each time node $t_n$ and evolves through a one-step numerical method over $[t_n, t_{n+1}]$ with initial datum $Z_n$. Within this framework, the dynamics governed by subsystem (\ref{eq:3}) are aligned with the $x$-axis, while those of subsystem (\ref{eq:4}) are aligned with the $y$-axis. Geometrically, this stochastic splitting method decomposes the phase flow into two orthogonal components along these coordinate axes. Building upon this decomposition, we implement the following two stochastic splitting methods.

\textbf{I. Stochastic Lie--Trotter splitting method}

When ${\varphi} _{t}^{1}$ and ${\varphi} _{t}^{2}$ represent the exact solutions of the two subsystems of the stochastic Hamiltonian system respectively, the stochastic Lie--Trotter splitting method is defined as
\begin{equation}\label{eq:16}
	Z_{n+1}={\varphi} _{h}^{1} \circ{\varphi} _{h}^{2} (Z_n).
\end{equation}
Substituting the exact solutions \eqref{eq1111} and \eqref{eq2222} into the above Eq. \eqref{eq:16}, we finally obtain the following method starting from $(X_0^\top,~Y_0^\top)^{\top}$ for the stochastic LV model \eqref{eq:2}
\begin{equation}\label{eq:5}
	\begin{split}
		& X_{n+1}=X_{n} {\ast}  \exp  \left\{\left(-\Gamma ^{(2)} Y_{n}+\eta^{(2)}-\frac{1}{2} \Lambda^{(2)}\right) h+\Sigma^{(2)}(W(t_{n+1})-W(t_n))\right\} ,\\[2mm]
		& Y_{n+1}=Y_{n} {\ast}  \exp  \left\{\left(\Gamma ^{(1)} X_{n+1}-\eta^{(1)}-\frac{1}{2} \Lambda^{(1)}\right) h+\Sigma^{(1)} \left(W(t_{n+1})-W(t_n)\right)\right\}.
	\end{split}
\end{equation}

\textbf{II. Stochastic Strang splitting method}

It follows from \eqref{eq1111}--\eqref{eq2222} that the stochastic Strang splitting method is defined as
\begin{equation}\label{eq:17}
	Z_{n+1}=	{\varphi} _{\frac{h}{2} }^{2} \circ{\phi} _{h}^{1} 	\circ{\varphi} _{\frac{h}{2} }^{2}(Z_n).
\end{equation}
Substituting the exact solutions \eqref{eq1111} and \eqref{eq2222} into the above Eq. \eqref{eq:17}, we finally obtain the following method starting from $(X_0^\top,~Y_0^\top)^{\top}$ for the stochastic LV model \eqref{eq:2}
\begin{equation}\label{eq:6}
	\begin{split}
		\widetilde{X}_n &=X_{n} {\ast}  \exp  \left\{\left(-\Gamma ^{(2)} Y_{n}+\eta^{(2)}-\frac{1}{2} \Lambda^{(2)}\right) {\frac{h}{2} }+\Sigma^{(2)}\left(W\left(t_{n}+{\frac{h}{2}}\right)-W(t_n)\right)\right\}, \\[2mm]
		Y_{n+1}&=Y_{n} {\ast}  \exp  \left\{\left(\Gamma ^{(1)} \widetilde{X}_{n}-\eta^{(1)}-\frac{1}{2} \Lambda^{(1)}\right) h+\Sigma^{(1)} \left(W(t_{n+1})-W(t_n)\right)\right\},\\[2mm]
		X_{n+1}&=\widetilde{X}_n {\ast}  \exp  \left\{\left(-\Gamma ^{(2)} Y_{n+1}+\eta^{(2)}-\frac{1}{2} \Lambda^{(2)}\right) \frac{h}{2}+\Sigma^{(2)}\left(W(t_{n+1})-W\left(t_n+{\frac{h}{2} }\right)\right)\right\} .
	\end{split}
\end{equation}

By decomposing the $2d$-dimensional stochastic LV model \eqref{eq:1} into computationally tractable low-dimensional subsystems (either two or three subsystems of dimension $d$), it can be seen that \eqref{eq:5} and \eqref{eq:6} can reduce the computational cost and improve the efficiency. Moreover, we note that these two stochastic splitting methods can preserve the positivity of the numerical solution $(X_n^\top,~Y_n^\top)^\top$, which stated in the following theorem.
\begin{theorem}\label{T1}
	For any deterministic initial value $(X_0^\top,~ Y_0^\top)^\top \in \mathbb{R}_{+}^{2d}$, the numerical solutions $(X_n^\top,~ Y_n^\top)^\top$ of  \eqref{eq:5} and \eqref{eq:6} both satisfy $(X_n^\top, ~Y_n^\top)^\top \in \mathbb{R}_{+}^{2d}$ for all $n \in \{1, \cdots, N\}$.
\end{theorem}
\begin{proof}
	First we consider the stochastic Strang splitting method \eqref{eq:6}. For $n=0,1,\cdots,N-1$, assume that $X_n>0$ and $Y_n>0$, then we can obtain directly
	\begin{equation}{\notag}
		Y_{n+1}=Y_{n} {\ast}  \exp  \left\{\left(\Gamma ^{(1)} \widetilde{X}_{n}-\eta^{(1)}-\frac{1}{2} \Lambda^{(1)}\right) h+\Sigma^{(1)} \left(W(t_{n+1})-W(t_n)\right)\right\}>0.
	\end{equation}
	Similarly, we have
	\begin{equation}{\notag}
		\widetilde{X}_n =X_{n} {\ast}  \exp  \left\{\left(-\Gamma ^{(2)} Y_{n}+\eta^{(2)}-\frac{1}{2} \Lambda^{(2)}\right) {\frac{h}{2} }+\Sigma^{(2)}\left(W\left(t_{n}+{\frac{h}{2}}\right)-W(t_n)\right)\right\}>0,
	\end{equation}
	and this implies
	\begin{equation}{\notag}
		X_{n+1}=\widetilde{X}_n {\ast}  \exp  \left\{\left(-\Gamma ^{(2)} Y_{n+1}+\eta^{(2)}-\frac{1}{2} \Lambda^{(2)}\right) \frac{h}{2}+\Sigma^{(2)}\left(W(t_{n+1})-W\left(t_n+{\frac{h}{2} }\right)\right)\right\}>0.
	\end{equation}
	Proceeding by induction, we arrive at the assertion of \eqref{eq:6}. By similar arguments, it can be shown that the numerical solution of \eqref{eq:5} satisfies $(X_n^\top,~ Y_n^\top)^\top \in \mathbb{R}_{+}^{2d}$.
\end{proof}
\subsection{Stochastic symplecticity}\label{sec:symplectic}
It is well-established that certain stochastic Hamiltonian systems inherently preserve geometric structures, notably stochastic symplecticity. Conventional numerical methods, such as the EM method, however, often fail to maintain these structures, resulting in a loss of long-term phase space fidelity. To address this, we derive in this section the discrete stochastic symplectic conservation law for the stochastic splitting methods introduced in Section \ref{sec:positivity}. Through this framework, we rigorously prove that both stochastic positivity-preserving Lie--Trotter and stochastic positivity-preserving Strang splitting methods preserve discrete stochastic symplectic conservation law. Given the methodological parallels between these methods, we present a detailed analysis exclusively for the stochastic Strang splitting method, with analogous arguments applying to stochastic Lie--Trotter splitting method.

\begin{theorem}\label{T2}
	Assume that the coefficient matrices $\Gamma^{(k)}$, $k = 1, 2$ are diagonal. Then \eqref{eq:3} and \eqref{eq:4} possess the stochastic symplectic conservation law.
\end{theorem}
\begin{proof}
	According to Lemma \ref{p3}, to demonstrate that the subsystems \eqref{eq:3} preserves the stochastic symplectic conservation law, it is necessary to prove 
	$${\left[\frac{\partial \varphi_{t}^{1}\left(Z_{0}\right)}{\partial Z_{0}}\right]^{\top} K\left(\varphi_{t}^{1}\left(Z_{0}\right)\right)\left[\frac{\partial \varphi_{t}^{1}\left(Z_{0}\right)}{\partial Z_{0}}\right] }=K(Z_0).$$
	The initial value is given by the vector $Y_0=({y_{1}^{0}},\ldots,{y_{d}^{0}})^{\top}$. Define the diagonal matrix 
	\begin{equation}\notag
		Y^{*}=\text{diag}\left\{\frac{y_{1}}{{y_{1}^{0}}},\ldots,\frac{y_{d}}{{y_{d}^{0}}}\right\},
	\end{equation}
	it yields that
	\begin{equation}\label{eq:14}
		\begin{aligned}
			& {\left[\frac{\partial \varphi_{t}^{1}\left(Z_{0}\right)}{\partial Z_{0}}\right]^{\top} K\left(\varphi_{t}^{1}\left(Z_{0}\right)\right)\left[\frac{\partial \varphi_{t}^{1}\left(Z_{0}\right)}{\partial Z_{0}}\right] } \\[2mm] 
			& =\left(\begin{array}{cc}
				I & 0 \\[2mm] 
				\Gamma ^{(1)} \overline{Y} t & Y^{*}
			\end{array}\right)^{\top}\left(\begin{array}{cc}
				0 & -K^{*}(z) \\[2mm] 
				K^{*}(z) & 0
			\end{array}\right)\left(\begin{array}{cc}
				I & 0 \\[2mm] 
				\Gamma ^{(1)} \overline{Y} t & Y^{*}
			\end{array}\right)\\[2mm]
			& =\begin{pmatrix}{{\left(\Gamma ^{(1)}\overline{Y}\right)^{\top}K^{*}(Z)t-K^{*}(Z)\Gamma ^{(1)}\overline{Y}}t} & {-K^{*}(Z)Y^{*}}\\[2mm] 
				{Y^{*}K^{*}(Z)} & {0}\end{pmatrix}.
		\end{aligned}
	\end{equation}
	Given that $\Gamma^{(1)}$, $K^{*}(Z)$, $Y^{*}$ and $\overline{Y}$ are all diagonal matrices, the Eq. \eqref{eq:14} can be rewritten as
	\begin{equation}\label{eq:25}
		\begin{aligned}
			{\left[\frac{\partial \varphi_{t}^{1}\left(Z_{0}\right)}{\partial Z_{0}}\right]^{\top} K\left(\varphi_{t}^{1}\left(Z_{0}\right)\right)\left[\frac{\partial \varphi_{t}^{1}\left(Z_{0}\right)}{\partial Z_{0}}\right] } & =\left(\begin{array}{cc}{0}&{-K^{*}(Z)Y^{*}}\\[2mm]
				{K^{*}(Z)Y^{*}}&{0}\end{array}\right),
		\end{aligned}
	\end{equation}
	which by  $K^{*}(Z)Y^{*}=K^{*}(Z_0)$, implies that \eqref{eq:25} is equivalent to
	\begin{equation}\label{eq:26}
		\begin{aligned}
			{\left[\frac{\partial \varphi_{t}^{1}\left(Z_{0}\right)}{\partial Z_{0}}\right]^{\top} K\left(\varphi_{t}^{1}\left(Z_{0}\right)\right)\left[\frac{\partial \varphi_{t}^{1}\left(Z_{0}\right)}{\partial Z_{0}}\right] } =\begin{pmatrix}{0}&{-K^{*}(Z_0)}\\[2mm]
				{K^{*}(Z_{0})}&{0}\end{pmatrix} =K(Z_0).
		\end{aligned}
	\end{equation}
	Therefore, it has been demonstrated that the phase flow of subsystem \eqref{eq:3} satisfies the stochastic symplectic conservation law. The proof for subsystem \eqref{eq:4} follows a similar approach and is thus omitted for brevity.
\end{proof}

In particular, when $t = h$ or $t = h/2$, the phase flows corresponding to the two subsystems continue to satisfy Eq. \eqref{eq:15}. Consequently, we can demonstrate that their composite operation also preserves stochastic symplectic conservation. This implies that the stochastic splitting methods \eqref{eq:5} and \eqref{eq:6} preserves the stochastic symplectic conservation law.
\begin{theorem}
	The stochastic Strang splitting method preserves the discrete stochastic symplectic conservation law
	\begin{equation}\label{eq:12}
		{dZ_{n+1}\wedge K(Z_{n+1})dZ_{n+1}=dZ_{n}\wedge K(Z_{n})dZ_{n}},\quad a.s.,
	\end{equation}
	where $Z_n=(X_n^\top,~Y_n^\top)^\top$ and $n\in{\mathbb{N}}$.
	
\end{theorem}
\begin{proof}
	Define $\widetilde{Z}_{n}:= \varphi_{\frac{h}{2}}^{2}(Z_{n})$ and $\hat{Z}_{n}:= \varphi_{h}^{1}\left(\widetilde{Z}_{n}\right)$. According to Lemma \ref{p3}, to demonstrate the assertion we need to prove
	\begin{equation}\label{eq:18}
		\left[{\frac{\partial Z_{n+1}}{\partial{Z_{n}}}}\right]^{\top}K\left(Z_{n+1}\right)\left[\frac{\partial Z_{n+1}}{\partial{Z_{n}}}\right]\\=K(Z_n).
	\end{equation}
	Note that
	\begin{equation}\label{eq:44}
		Z_{n+1}=	{\varphi} _{\frac{h}{2} }^{2} \circ{\varphi} _{h}^{1} 	\circ{\varphi} _{\frac{h}{2} }^{2}(Z_n).
	\end{equation}
	Substituting $\widetilde{Z}_{n} = \varphi_{\frac{h}{2}}^{2}(Z_{n})$, $\hat{Z}_{n} = \varphi_{h}^{1}\left(\widetilde{Z}_{n}\right)$ and (\ref{eq:44}) into the left-hand side of (\ref{eq:18}), then (\ref{eq:18}) can be expressed as 
	\begin{equation}\label{eq:19}
		\left[\frac{\partial \varphi_{\frac{h}{2} }^{2}\left(\hat{Z}_{n}\right)}{\partial \hat{Z}_{n}}\frac{\partial \varphi_{h}^{1}\left(\widetilde{Z}_{n}\right)}{\partial \widetilde{Z}_{n}}\frac{\partial \varphi_{\frac{h}{2}}^2\left(Z_{n}\right)}{\partial Z_{n}}\right]^{\top}K\left(\varphi_{\frac{h}{2} }^{2}\left(\hat{Z}_{n}\right)\right)\left[\frac{\partial\varphi_{\frac{h}{2} }^{2}\left(\hat{Z}_{n}\right)}{\partial \hat{Z}_{n}}\frac{\partial\varphi_{h}^{1}\left(\widetilde{Z}_{n}\right)}{\partial \widetilde{Z}_{n}}\frac{\partial\varphi_{\frac{h}{2} }^{2}(Z_{n})}{\partial Z_{n}}\right]=K(Z_n).
	\end{equation}
	From this, it follows that \eqref{eq:19} is equivalent to
	\begin{equation}\label{eq:20}
		\begin{split}
			\left[\frac{\partial\varphi_{\frac{h}{2} }^{2}({Z}_{n})}{\partial{Z}_{n}}\right]^{\top}&\left[\frac{\partial\varphi_{h}^1\left(\widetilde{Z}_{n}\right)}{\partial \widetilde{Z}_{n}}\right]^{\top}\left(\left[\frac{\partial \varphi_{\frac{h}{2} }^{2}\left(\hat{Z}_{n}\right)}{\partial \hat{Z}_{n}}\right]^{\top}K\left(\varphi_{\frac{h}{2} }^{2}\left(\hat{Z}_{n}\right)\right)\left[\frac{\partial\varphi_{\frac{h}{2} }^{2}\left(\hat{Z}_{n}\right)}{\partial \hat{Z}_{n}}\right]\right)\\[2mm]
			&\cdot\left[\frac{\partial\varphi_{h}^{1}\left(\widetilde{Z}_{n}\right)}{\partial \widetilde{Z}_{n}}\right]\left[\frac{\partial\varphi_{\frac{h}{2} }^{2}({Z}_{n})}{\partial Z_{n}}\right]=K(Z_n).
		\end{split}
	\end{equation}
	Since $\varphi_{\frac{h}{2} }^{2}$ is the exact solution of  \eqref{eq:4}, according to Theorem \ref{T2}, we get
	\begin{equation}\label{eq:21}
		\left[\frac{\partial \varphi_{\frac{h}{2} }^{2}\left(\hat{Z}_{n}\right)}{\partial \hat{Z}_{n}}\right]^{\top}K\left(\varphi_{\frac{h}{2} }^{2}\left(\hat{Z}_{n}\right)\right)\left[\frac{\partial\varphi_{\frac{h}{2} }^{2}\left(\hat{Z}_{n}\right)}{\partial \hat{Z}_{n}}\right]=
		K\left(\hat{Z}_{n}\right).
	\end{equation}
	Substituting (\ref{eq:21}) into (\ref{eq:20}), then we need to demonstrate
	\begin{equation}\label{eq:22}
		\left[\frac{\partial\varphi_{\frac{h}{2} }^{2}({Z}_{n})}{\partial{Z}_{n}}\right]^{\top}\left(\left[\frac{\partial\varphi_{h}^1\left(\widetilde{Z}_{n}\right)}{\partial \widetilde{Z}_{n}}\right]^{\top}K\left( \varphi_{h}^{1}\left(\widetilde{Z}_{n}\right)\right)\left[\frac{\partial\varphi_{h}^{1}\left(\widetilde{Z}_{n}\right)}{\partial \widetilde{Z}_{n}}\right]\right)\left[\frac{\partial\varphi_{\frac{h}{2} }^{2}({Z}_{n})}{\partial Z_{n}}\right]=K(Z_n).
	\end{equation}
	Similarly, it follows from $\varphi_h^{1}$ is the exact solution of \eqref{eq:3} that
	\begin{equation}\label{eq:23}
		\left[\frac{\partial\varphi_{h}^1\left(\widetilde{Z}_{n}\right)}{\partial \widetilde{Z}_{n}}\right]^{\top}K\left( \varphi_{h}^{1}\left(\widetilde{Z}_{n}\right)\right)\left[\frac{\partial\varphi_{h}^{1}\left(\widetilde{Z}_{n}\right)}{\partial \widetilde{Z}_{n}}\right]=K\left(\widetilde{Z}_{n} \right).
	\end{equation}
	Substituting \eqref{eq:23} into \eqref{eq:22}, then we need to establish 
	\begin{equation}\label{eq:24}
		\left[\frac{\partial\varphi_{\frac{h}{2} }^{2}({Z}_{n})}{\partial{Z}_{n}}\right]^{\top}K\left( \varphi_{\frac{h}{2} }^{2}\left({Z}_{n}\right)\right)\left[\frac{\partial\varphi_{\frac{h}{2} }^{2}({Z}_{n})}{\partial Z_{n}}\right]
		=K\left( Z_n\right).
	\end{equation}
	Since $\varphi_{\frac{h}{2} }^{2}$ is the exact solution of the subsystem (\ref{eq:4}), according to Theorem \ref{T2}, it can be proved that (\ref{eq:24}) holds. 
\end{proof}
The proof idea of the stochastic Strang splitting method can be used. Similarly, it can be shown that the stochastic  Lie--Trotter splitting method preserves discrete stochastic symplectic conservation law.
\begin{theorem}
	The stochastic Lie--Trotter splitting method preserves the discrete stochastic symplectic conservation law
	\begin{equation}{\notag}
		{dZ_{n+1}\wedge K(Z_{n+1})dZ_{n+1}=dZ_{n}\wedge K(Z_{n})dZ_{n}},\quad a.s.,
	\end{equation}
	where $Z_n=(X_n^\top,~Y_n^\top)^\top$ and $n\in{\mathbb{N}}$.
\end{theorem}
\subsection{Uniform boundedness of the numerical solution}\label{sec:bound}
In this subsection, we will study the uniform boundedness of the numerical solutions of \eqref{eq:5} and \eqref{eq:6} in the case of $\Sigma^{(2)} \equiv 0$.
\begin{theorem}\label{T4}
	Given any deterministic initial value $(X_0^\top,~ Y_0^\top)^\top \in \mathbb{R}_{+}^{2d}$, for any $p\geq1$, the numerical solution $X_n\in \mathbb{R}_{+}^{d}$ of the stochastic Strang splitting method \eqref{eq:6} is uniformly bounded
	$$\underset{n=1,\cdots,N}{\sup} \mathbb{E}\big[|X_n|^p\big]{\le}  C .$$
	Moreover, when $\Sigma^{(2)}\equiv0$, the $p$-th moment of $Y_n\in \mathbb{R}_{+}^{d}$ is also uniformly bounded
	$$	\underset{n=1,\cdots,N}{\sup} \mathbb{E}\big[|Y_n|^p\big]{\le}  C,$$
	where the positive constant $C=C(X_0, Y_0,\Gamma^{(1)},\Gamma^{(2)},\eta^{(1)},\eta^{(2)},\Sigma^{(1)},\Sigma^{(2)})$.
\end{theorem} 
\begin{proof}
	First we consider the uniformly boundedness of the numerical solution $X_{n}$. For any $n=0,1,\cdots, N-1$, according to the stochastic Strang splitting method \eqref{eq:6}, we have
	\begin{align*}
		\widetilde{X}_{n}=X_{n} * \exp \left\{\left(-\Gamma ^{(2)} Y_{n}+\eta^{(2)}-\frac{1}{2} \Lambda^{(2)}\right)  \frac{h}{2}+\Sigma ^{(2)}\left(W\left(t_{n}+\frac{h}{2}\right)-W\left(t_{n}\right)\right)\right\},
	\end{align*}
	and
	\begin{align}\label{eq:28}
		X_{n+1}=\widetilde{X}_{n} & * \exp \left\{\left(-\Gamma ^{(2)}Y_{n+1}+\eta^{(2)}-\frac{1}{2}\Lambda^{(2)}\right) \frac{h}{2}+\Sigma ^{(2)}\left(W\left(t_{n+1}\right)-W\left(t_{n}+\frac{h}{2}\right)\right)\right\}.
	\end{align}
	Since $Y_n$, $Y_{n+1}$, $\Gamma^{(2)}$ and $\Lambda^{(2)}$ are all positive, we can deduce
	\begin{align}\label{eq:31}
		\widetilde{X}_{n}\leqslant X_{n} * \exp \left\{\eta^{(2)}  \frac{h}{2}+\Sigma ^{(2)}\left(W\left(t_{n}+\frac{h}{2}\right)-W\left(t_{n}\right)\right)\right\}.
	\end{align}
	and
	\begin{align}\label{eq:48}
		{X}_{n+1}\leqslant \widetilde{X}_{n} * \exp \left\{\eta^{(2)}  \frac{h}{2}+\Sigma ^{(2)}\left(W\left(t_{n+1}\right)-W\left(t_{n}+\frac{h}{2}\right)\right)\right\}.
	\end{align}
	Substituting \eqref{eq:31} into \eqref{eq:48}, we have
	\begin{align*}
		X_{n+1}
		\leqslant X_{n}  * \exp \left\{\eta^{(2)}  h+\Sigma ^{(2)}\left(W\left(t_{n+1}\right)-W\left(t_{n}\right)\right)\right\},
	\end{align*}
	which yields
	\begin{align*}
		X_{n+1}
		\leqslant X_{0}  * \exp \left\{\eta^{(2)}  T+\Sigma ^{(2)}W\left(t_{n+1}\right)\right\}
		\overset{\triangle}= X_{0}  * C_1* \exp \left\{\Sigma ^{(2)}W\left(t_{n+1}\right)\right\}.
	\end{align*}
	Note that, in view of $|U*V|{\le}|U||V|$ for any $U,~V\in \mathbb{R}^d$, we know that
	\begin{equation}{\label{eq:32}}
		\begin{split}
			\mathbb{E}\left|X_{n+1}\right|^{p}  
			& \le\left|X_{0}\right|^{p} \left|C_1\right|^{p} \mathbb{E}\left[\left|\exp \left( \Sigma ^{(2)} W\left(t_{n+1}\right)\right)\right|^{p}\right]\\
			&=\left|X_{0}\right|^{p} \left|C_1\right|^{p} \mathbb{E}\left[ \sum_{i=1}^{d}\left(\exp \left(\sum_{j=1}^{m} \sigma_{i j}^{(2)} W_{j}\left(t_{n+1}\right)\right)\right)^2\right]^{\frac{p}{2}}
		\end{split}
	\end{equation}
	It follows from \cite[Theorem 3.3]{Hong} that the inequality \eqref{eq:32} can be bounded as
	\begin{equation}\label{eq:33}
		\begin{split}
			\mathbb{E}\left|X_{n+1}\right|^{p}  
			&	\lesssim \sum_{i=1}^{d}\left|X_{0}\right|^{p} \left|C_1\right|^{p} \mathbb{E}\left[\exp \left(p \sum_{j=1}^{m} \sigma_{i j}^{(2)} W_{j}\left(t_{n+1}\right)\right)\right]\\
			& = \sum_{i=1}^{d}\prod_{j=1}^{m}\left|X_{0}\right|^{p} \left|C_1\right|^{p} \mathbb{E}\left[\exp \left(p  \sigma_{i j}^{(2)} W_{j}\left(t_{n+1}\right)\right)\right].
		\end{split}
	\end{equation}
	Since
	\begin{align}\label{eq:34}
		\mathbb{E} \big[ \exp \left(C W_j\left(t\right)\right) \big]=\exp\left(\frac{1}{2} C^{2}t\right),{\quad}j=1,2,\cdots,m,
	\end{align}
	by plugging $C=p\sigma_{i j}^{(2)}$ into (\ref{eq:34}), we derive from (\ref{eq:33}) that
	\begin{align}\notag
		\mathbb{E}\left|X_{n+1}\right|^{p}  	 \le\sum_{i=1}^{d} \prod_{j=1}^{m}\left|X_{0}\right|^{p} \left|C_1\right|^{p}\exp \left(\frac{1}{2} p^{2}\left(\sigma_{i j}^{(2)}\right)^{2} t_{n+1}\right),
	\end{align}
	which implies
	\begin{align}\notag
		\mathbb{E}\left|X_{n+1}\right|^{p}  	\le\sum_{i=1}^{d} \prod_{j=1}^{m}\left|X_{0}\right|^{p} \left|C_1\right|^{p}\exp \left(\frac{1}{2} p^{2}\left(\sigma_{i j}^{(2)}\right)^{2} T\right) \overset{\triangle}=C.
	\end{align}
	This proves the the uniform boundedness of \( \mathbb{E}\big[\left|X_{n}\right|^{p}\big] \).
	
	Now we are in the position to prove the the uniform boundedness of \( \mathbb{E}\big[\left|Y_{n}\right|^{p}\big] \). When the coefficient matrix $\Sigma^{(2)}\equiv0$, for any $n=0, 1, \cdots, N - 1$, due to the stochastic Strang splitting method \eqref{eq:6}, we get
	\begin{align*}
		Y_{n+1}  =Y_{n} * \exp \left\{\left(\Gamma^{(1)} \widetilde{X}_{n}-\eta^{(1)}-\frac{1}{2} \Lambda^{(1)}\right) h+\Sigma^{(1)}\left(W\left(t_{n+1}\right)-W\left(t_{n}\right)\right)\right\}.
	\end{align*}
	Notice that $\eta^{(1)}$ and $\Lambda^{(1)}$ are all positive, then
	\begin{align}\label{eq:36}
		Y_{n+1}\leq Y_{n} * \exp  \left\{\Gamma^{(1)} \widetilde{X}_{n} h+\Sigma^{(1)}\left(W\left(t_{n+1}\right)-W\left(t_{n}\right)\right)\right\} .
	\end{align}
	Substituting \eqref{eq:31} into \eqref{eq:36}, it holds
	\begin{equation}\label{eq:37}
		\begin{split}
			Y_{n+1}&\leq Y_{n} * \exp \left\{\Gamma^{(1)} hX_{n} * \exp \left(\eta^{(2)} \frac{h}{2}\right) +\Sigma^{(1)}\left(W\left(t_{n+1}\right)-W\left(t_{n}\right)\right)\right\}\\
			&\leq Y_{n} * \exp \left\{\Gamma^{(1)}h X_{0} * \exp \left(\eta^{(2)} T+\eta^{(2)} \frac{h}{2}\right)+\Sigma^{(1)}\left(W\left(t_{n+1}\right)-W\left(t_{n}\right)\right)\right\}\\
			&\leq Y_{0} * \exp \left\{\Gamma^{(1)}T X_{0} * \exp \left(\eta^{(2)} T+\eta^{(2)} \frac{h}{2}\right)+\Sigma^{(1)}W\left(t_{n+1}\right)\right\}\\\notag
			&\overset{\triangle}= Y_{0} * C_{2} * \exp \left\{\Sigma^{(1)}W\left(t_{n+1}\right)\right\}.{\notag}
		\end{split}
	\end{equation}
	Therefore, we have
	\begin{align}
		\mathbb{E}\big[\left|Y_{n+1}\right|^{p} \big] \le\left|Y_{0}\right|^{p}\left|C_{2}\right|^{p} \mathbb{E}\left[\left|\exp \left\{\Sigma^{(1)}W\left(t_{n+1}\right)\right\}\right|^{p}\right].
	\end{align}
	Analogous to the proof of $\mathbb{E}\big[\left|X_{n}\right|^{p}\big]$, we obtain
	\begin{align}{\notag}
		\mathbb{E}\big[\left|Y_{n+1}\right|^{p}\big] \leqslant \sum_{i=1}^{d} \prod_{j=1}^{m}\left|Y_{0}\right|^{p}\left|C_{2}\right|^{p} \exp \left(\frac{1}{2} p^{2}\left(\sigma_{i j}^{(1)}\right)^{2} T\right) \overset{\triangle}=C.{\notag}
	\end{align}
	Consequently, this establishes the uniform boundedness of \( \mathbb{E}\big[\left|Y_{n}\right|^{p}\big] \).
\end{proof}
For the stochastic Lie--Trotter splitting method \eqref{eq:5}, we can derive the following uniform boundedness of the numerical solution. Proof is similar to that of Theorem \ref{T4}.
\begin{theorem}\label{T7}
	Given any deterministic initial value $(X_0^\top,~ Y_0^\top)^\top \in \mathbb{R}_{+}^{2d}$, for any  $p\geq1$, the numerical solution $X_n\in \mathbb{R}_{+}^{d}$ of the stochastic Lie-Trotter splitting method \eqref{eq:5} is uniformly bounded
	$$\underset{n=1,\cdots,N}{\sup} \mathbb{E}\big[|X_n|^p\big]{\le}  C .$$
	Moreover, when $\Sigma^{(2)}\equiv0$, the $p$-th moment of $Y_n\in \mathbb{R}_{+}^{d}$ is also uniformly bounded
	$$	\underset{n=1,\cdots,N}{\sup} \mathbb{E}\big[|Y_n|^p\big]{\le}  C,$$
	where the positive constant $C=C(X_0, Y_0,\Gamma^{(1)},\Gamma^{(2)},\eta^{(1)},\eta^{(2)},\Sigma^{(1)},\Sigma^{(2)})$.
\end{theorem} 
\begin{remark}
	It follows from the procedure of the proofs of the aforementioned theorems, we can observe that when $\Sigma^{(2)} \not\equiv 0$, the $p$-th moment of $X_n$ is uniformly bounded. However, it cannot be guaranteed that the $p$-th moment of $Y_n$ is also uniformly bounded.
\end{remark}

Therefore, we can derive the uniform boundedness of the solutions $(X_n^\top,~Y_n^\top)^\top \in \mathbb{R}_{+}^{2d}$ of \eqref{eq:5} and \eqref{eq:6}.
\begin{theorem}\label{T6}
	Assume that $\Sigma^{(2)}\equiv0$. Given any deterministic initial value $(X_0^\top,~ Y_0^\top)^\top \in \mathbb{R}_{+}^{2d}$, for any  $p\geq1$, the numerical solutions  $(X_n^\top,~Y_n^\top)^\top \in \mathbb{R}_{+}^{2d}$ of the stochastic Lie--Trotter splitting method \eqref{eq:5} and the stochastic Strang splitting method \eqref{eq:6} are both uniformly bounded
	$$	\underset{n=1,\cdots,N}{\sup} \mathbb{E}\left[|X_n|^p+|Y_n|^p\right]{\le}  C,$$
	where the positive constant $C=C(X_0, Y_0,\Gamma^{(1)},\Gamma^{(2)},\eta^{(1)},\eta^{(2)},\Sigma^{(1)},\Sigma^{(2)})$.
\end{theorem}
\subsection{Strong convergence order}\label{sec:convergence order}
This section is concerned with the convergence analysis of two stochastic positivity-preserving symplectic methods proposed in the previous subsection.

For the stochastic LV model \eqref{eq:1}, we denote by $Z_S|_{t_0,Z_0}(t_k)$ the approximate value at time $t_k$ which obtained after $k$ iterations of stochastic splitting methods on the interval $[t_0, T]$, starting from the initial condition $Z_0 = Z(t_0)$. $Z|_{t_0,Z_0}(t_k)$ represents the exact solution at time $t_k$. 

The following fundamental theorem on the mean-square order of convergence is very useful in applications; see for instance \cite[Theorem 1.1]{M}.
\begin{theorem}\label{T5}
	Suppose the one-step approximation $Z_S|_{t_0,Z_0}(t_k)$ has order of accuracy $p_1$ for expectation of the deviation and order of accuracy $p_2$ for the meansquare deviation; more precisely, for arbitrary to  $t_0 \leq t \leq t_0 + T - h$, the following inequalities hold:
	\begin{align*}
		\left|\mathbb{E}\left[Z|_{t,Z_n}(t+h)-Z_S|_{t,Z_n}(t+h)\right]\right| & \leq K\left(1+|Z_n|^{2}\right)^{1 / 2} h^{p_{1}},
	\end{align*}
    \begin{align*}
		{\left[\mathbb{E}\left|Z|_{t,Z_n}(t+h)-Z_S|_{t,Z_n}(t+h)\right|^2\right]}^{1/2}   \leq K\left(1+|Z_n|^{2}\right)^{1 / 2} h^{p_{2}}.
	\end{align*}
	Also let
	$$p_{2} {\ge}\frac{1}{2},{\quad}p_{1} {\ge} p_{2}+\frac{1}{2} .$$
	Then for any $N$ and $k = 0, 1, \cdots, N$, the following inequality holds
	\[ \left[ \mathbb{E}  \left| Z|_{t_0,Z_0}(t_k) - Z_S|_{t_0,Z_0}(t_k) \right|^2 \right]^{1/2} \leq K \left( 1 + \mathbb{E}|Z_0|^2 \right)^{1/2} h^{p_2 - 1/2}. \]
	i.e., the order of accuracy of the method constructed using the one-step approximation $Z_S|_{t,Z_n}(t+h)$ is $p = p_2 - 1/2$.
\end{theorem}

Next, through Theorem \ref{T5}, the relationship between the local error order and the global error order of the stochastic Strang splitting method \eqref{eq:6} in the $L^2(\Omega)$-norm is established.
\begin{theorem}\label{thm3.9}
	Let conditions in Theorem \ref{T6} hold. The stochastic positivity-preserving Strang symplectic splitting method \eqref{eq:6} converges with global order one in the $L^2(\Omega)$-norm.
\end{theorem}
\begin{proof}
	Since the proof of the mathematical expectation accuracy for the one-step approximation deviation starting from any arbitrary time is similar to that starting from the initial value $Z_0$, we focus specifically on the proof of the accuracy starting from the initial value $Z_0$. 
	
	Let the numerical solution of the one-step approximation at $t = h$ be denoted as $Z_1 = (X_1^\top, Y_1^\top)^{\top}$ and let the exact solution at $t = h$ be denoted as $Z(h) = \left(X^{\top}(h), Y^{\top}(h)\right)^{\top}$. 
	
	First, the local error order of the one-step approximation for the quantity $Y_1$ is determined. The expressions for $Y(h)$ and $Y_1$ are given in the following integral forms
	\begin{equation}\label{eq:7}
		Y(h)=Y_{0}+\int_{0}^{h} Y(t)* \left(\Gamma ^{(1)} X(t)-\eta^{(1)}-\frac{1}{2} \Lambda^{(1)}\right) d t+\int_{0}^{h} \Sigma^{(1)} Y(t) d W(t) ,
	\end{equation} 
	and
	\begin{equation}\label{eq:8}
		Y_{1}=Y_{0}+\int_{0}^{h} Y(t)*\left(\Gamma ^{(1)} \widetilde{X}_{0}-\eta^{(1)}-\frac{1}{2} \Lambda^{(1)} \right)d t+\int_{0}^{h} \Sigma^{(1)} Y(t) d W(t).	
	\end{equation} 
	The difference between the exact solution $Y(h)$ and the numerical solution $Y_{1}$ is denoted as
	\begin{equation}\label{eq:9}
		Y(h)-Y_{1}=\Gamma ^{(1)} \int_{0}^{h} Y(t)*\left(X(t)-\widetilde{X}_{0}\right) dt,
	\end{equation}
	where
	\begin{equation}\label{eq:10}
		X(t)=X_{0}+\int_{0}^{t} X(s)*\left(-\Gamma^{(2)} Y{(s)}+\eta^{(2)}\right) ds,
	\end{equation}
	\begin{equation}\label{eq:11}
		\widetilde{X}_{0}=X_{0}+\int_{0}^{\frac{h}{2}} X(s)*\left(-\Gamma^{(2)} Y_{0}+\eta^{(2)}\right) ds .
	\end{equation}
	By substituting (\ref{eq:10}) and (\ref{eq:11}) into (\ref{eq:9}), it is evident that the deviation of $Y_1$ can be represented by a double integral,
	\begin{align*}
		&\left| \mathbb{E}\left[Y(h)-Y_1\right]\right|\\[2mm]
		& =\left|\Gamma ^{(1)} \int_{0}^{h} \mathbb{E} \left[ Y(t)*\left(\int_{0}^{t} X(s)*\left(-\Gamma^{(2)} Y{(s)}+\eta^{(2)}\right) d s-\int_{0}^{\frac{h}{2}} X(s)*\left(-\Gamma^{(2)} Y_{0}+\eta^{(2)}\right) d s\right) \right]d t\right|.
	\end{align*}
	For any matrix $A\in\mathbb{R}^{d\times d}$ and vector $U\in\mathbb{R}^{d}$, we know that $\left|AU\right|\le|A|_F|U|$, where $|\cdot|_F$ is Frobenius norm of a matrix. Furthermore, by using the uniform boundedness of $X(t)$, $Y(t)$ given by Lemma \ref{p2}, we can obtain
	\begin{align*}
		\left| \mathbb{E}\left[Y(h)-Y_1\right]\right|\le &\left|\Gamma ^{(1)}\right|_F\int_{0}^{h}\int_{0}^{t}\mathbb{E}\left[\left|Y(t)\right||\left|X(s)\right|\left|-\Gamma ^{(2)}Y(s)+\eta^{(2)}\right|\right]dsdt\\[2mm]
		& +\left|\Gamma ^{(1)}\right|_F\int_{0}^{h}\int_{0}^{\frac{h}{2}}\mathbb{E}\left[\left|Y(t)\right|\left|X(s)\right|\left|-\Gamma ^{(2)}Y_{0}+\eta^{(2)}\right|\right]dsdt\\[2mm]
		{\leq}& Ch^{2}.
	\end{align*}
	
	To assess the accuracy of the one-step approximation mean square deviation for the component $Y_1$, the Cauchy-Schwarz inequality is applied to the mean square deviation, thereby 
	\begin{equation}\label{eq:39}
		\begin{split}
			& \mathbb{E}\left|Y(h)-Y_1\right|^2 \\[2mm]
			& \leq 2h\left|\Gamma ^{(1)}\right|_F^{2}\int_{0}^{h}\mathbb{E}\left|\int_{0}^{t}Y(t)*X(s)*(-\Gamma^{(2)}Y(s)+\eta^{(2)})ds\right|^{2}dt\\
			& \quad{+}2h\left|\Gamma ^{(1)}\right|_F^{2}\int_{0}^{h}\mathbb{E}\left|\int_{0}^{\frac{h}{2} }Y(t)*X(s)*(-\Gamma^{(2)}Y_{0}+\eta^{(2)})ds\right|^{2}dt .
		\end{split}
	\end{equation}   
	Using the Cauchy-Schwarz inequality once more on (\ref{eq:39}) and utilizing the uniform boundedness of $X(t)$, $Y(t)$ given by Lemma \ref{p2} , it follows that
	\begin{align*}
		\mathbb{E}\left|Y(h)-Y_1\right|^2 & \leq 2h\left|\Gamma ^{(1)}\right|_F^{2}\int_{0}^{h}t\int_{0}^{t}\mathbb{E}\left[\left||Y(t)||X(s)||-\Gamma^{(2)}Y(s)+\eta^{(2)}|\right|\right]^{2}dsdt \\[2mm]
		& \quad{+}h^{2} \left|\Gamma ^{(1)}\right|_F^{2}\int_{0}^{h}\int_{0}^{\frac{h}{2} }\mathbb{E}\left[\left||Y(t)||X(s)||-\Gamma^{(2)}Y_{0}+\eta^{(2)}|\right|\right]^2dsdt \\[2mm]
		& \leq Ch^3 .
	\end{align*}   
	In conclusion, according to the fundamental theorem on the mean-square order of convergence stated in Theorem \ref{T5}, the global mean square error order of $Y_n$ is one, that is
	$$ \left[\mathbb{E}\left|Y(t_n)-Y_n\right|^{2}\right]^{\frac{1}{2}}\leq Ch.$$
	
	To determine the local error order of the first-order approximation for the numerical solution $X_1$, the expressions for $X(h)$ and $X_1$ are formulated in the following integral forms
	\begin{align}\label{eq:45}
		X(h)=X_{0}+\int_{0}^{h}X(t)*(-\Gamma ^{(2)}Y(t)+\eta ^{(2)})dt,
	\end{align}
	and
	\begin{align}\label{eq:40}
		X_1 =\widetilde{X}_{0}+\int_{\frac{h}{2}}^{h}X(t)*(-\Gamma^{(2)}Y_1+\eta^{(2)})dt.
	\end{align}
	Inserting (\ref{eq:10}) into (\ref{eq:40}), we derive
	\begin{align}\label{eq:41}
		X_1=X_{0}+\int_{0}^{\frac{h}{2}}X(t)*(-\Gamma ^{(2)}Y_0+\eta^{(2)})dt+\int_{\frac{h}{2}}^{h}X(t)*(-\Gamma ^{(2)}Y_1+\eta^{(2)})dt.
	\end{align}
	Subtracting \eqref{eq:45} from \eqref{eq:41}, we arrive at
	\begin{align}\label{eq:46}
		\big|\mathbb{E}\left[X(h)-X_1\right]\big|  =\left|\mathbb{E}\left[\int_{0}^{\frac{h}{2}}\Gamma ^{(2)}X(t)*(Y_{0}-Y(t))dt+\int_{\frac{h}{2}}^{h}\Gamma^{(2)}X(t)*(Y_1-Y(t))dt\right] \right|,
	\end{align}
	where
	\begin{equation}\label{eq:47}
		Y(t)=Y_{0}+\int_{0}^{t} Y(s)*  \left(\Gamma ^{(1)} X(s)-\eta^{(1)}-\frac{1}{2} \Lambda^{(1)}\right) d s+\int_{0}^{t} \Sigma^{(1)} Y(s) d W(s) .
	\end{equation} 
	By substituting (\ref{eq:8}) and (\ref{eq:47}) into (\ref{eq:46}), (\ref{eq:46}) is equivalent to
	\begin{align*}
		&\left|\mathbb{E}\left[X(h)-X_1\right]\right| \\[2mm]
		& = \left|\mathbb{E}\left[\int_{0}^{h} \int_{0}^{t}\Gamma^{(2)}X(t)*Y(s)*\left(\left(\Gamma^{(1)}X(s)-\eta ^{(1)}-\frac{1}{2}\Lambda^{(1)}\right)ds+\Sigma ^{(1)}dW(s)\right)dt\right.\right. \\[2mm]
		& \quad+\left.\left.\int_{{\frac{h}{2}}}^h \int_{0}^{h}\Gamma^{(2)}X(t)*Y(s)*\left(\left(\Gamma^{(1)}\widetilde{X}_{0}-\eta ^{(1)}-\frac{1}{2}\Lambda^{(1)}\right)ds+\Sigma ^{(1)}dW(s)\right)dt\right]\right|.
	\end{align*}
	Similarly, thanks to the fact that $|AU|\le|A|_F|U|$, we have
	\begin{align*}
		&\left|\mathbb{E}\left[X(h)-X_1\right]\right| \\[2mm]
		& \le \left|\Gamma ^{(2)}\right|_F\left|\mathbb{E}\left[\int_{0}^{h} \int_{0}^{t}X(t)*Y(s)*\left(\left(\Gamma^{(1)}X(s)-\eta ^{(1)}-\frac{1}{2}\Lambda^{(1)}\right)ds+\Sigma ^{(1)}dW(s)\right)dt\right.\right. \\[2mm]
		& \quad+\left.\left.\int_{{\frac{h}{2}}}^h \int_{0}^{h}X(t)*Y(s)*\left(\left(\Gamma^{(1)}\widetilde{X}_{0}-\eta ^{(1)}-\frac{1}{2}\Lambda^{(1)}\right)ds+\Sigma ^{(1)}dW(s)\right)dt\right]\right|\\
		&\leq\left|\Gamma^{(2)}\right|_F\int_{0}^{h}\int_{0}^{t}\mathbb{E}\left[|X(t)||Y(s)|\left|\Gamma^{(1)}X(s)-\eta^{(1)}-\frac{1}{2}\Lambda^{(1)}\right|\right]dsdt \\[2mm]
		&\quad +\left|\Gamma^{(2)}\right|_F\int_{\frac{h}{2}}^{h}\int_{0}^{h}\mathbb{E}\left[\left|X(t)\right|\left|Y(s)\right|\left|\Gamma^{(1)}\widetilde{X}_{0}-\eta^{(1)}-\frac{1}{2}\Lambda^{(1)}\right|\right]dsdt \\[2mm]
		& \leq Ch^{2},
	\end{align*}
	due to
	$$\mathbb{E}\left[\int_{0}^{t} CdW_j(s)\right]=0,\quad j=1,2,\cdots,m.$$ 
	When evaluating the mean square deviation of $X_1$, we know
	\begin{align}{\notag}
		& \mathbb{E}\mid X(h)-X_1\mid^{2} \\{\notag}
		& \le \left|\Gamma ^{(2)}\right|_F^2\mathbb{E}\left|\int_{0}^{h} \int_{0}^{t}X(t)*Y(s)*\left(\left(\Gamma^{(1)}X(s)-\eta ^{(1)}-\frac{1}{2}\Lambda^{(1)}\right)ds+\Sigma ^{(1)}dW(s)\right)dt\right. \\{\notag}
		&\quad +\left.\int_{{\frac{h}{2}}}^h \int_{0}^{h}X(t)*Y(s)*\left(\left(\Gamma^{(1)}\widetilde{X}_{0}-\eta ^{(1)}-\frac{1}{2}\Lambda^{(1)}\right)ds+\Sigma ^{(1)}dW(s)\right)dt\right|^2 .
	\end{align}
	The Cauchy--Schwarz inequality is initially applied to bound the mean square deviation, then
	\begin{align*}
		& \mathbb{E}\mid X(h)-X_1\mid^{2} \\[2mm]
		& \leq 4h\left|\Gamma ^{(2)}\right|_F^2\int_{0}^{h}t\int_{0}^{t}\mathbb{E}\left[|X(t)||Y(s)|\left|\Gamma^{(1)}X(s)-\eta^{(1)}-\frac{1}{2}\Lambda^{(1)}\right|\right]^2dsdt \\
		& \quad+4h\left|\Gamma ^{(2)}\right|_F^2\int_{0}^{h}\mathbb{E}\left|\int_{0}^{t}X(t)*Y(s)*\Sigma ^{(1)}dW(s)\right|^{2}dt \\{\notag}
		& \quad+2h\left|\Gamma ^{(2)}\right|_F^2\int_{\frac{h}{2}}^{h}t\int_{0}^{h}\mathbb{E}\left[X(t)||Y(s)|\left|\Gamma^{(1)}\widetilde{X}_{0}-\eta^{(1)}-\frac{1}{2}\Lambda^{(1)}\right|\right]^2dsdt \\
		&\quad +2h\left|\Gamma ^{(2)}\right|_F^2\int_{\frac{h}{2}}^{h}\mathbb{E}\left|\int_{0}^{h}X(t)*Y(s)*\Sigma ^{(1)}dW(s)\right|^{2}dt \\
		& \leq Ch^{3}.
	\end{align*}
	According to Theorem \ref{T5} again, the global mean square error order of $X_n$ is one, that is
	$$ \left[\mathbb{E}\left|X(t_n)-X_n\right|^{2}\right]^{\frac{1}{2}}\leq Ch.$$
	Therefore, for all $t_n = nh \in [0, T]$, it holds that
	\begin{align}{\notag}
		\left[ \mathbb{E}\left(\left|X(t_n)-X_n\right|^2+\left|Y(t_n)-Y_n\right|^2\right)\right]^{\frac{1}{2}}\leq Ch.{\notag}
	\end{align}
	The proof of the theorem is thus completed.
\end{proof}
Similar to Theorem \ref{thm3.9}, we have the following theorem.
\begin{theorem}
	Let conditions in Theorem \ref{T6} hold. The stochastic positivity-preserving Lie--Trotter symplectic splitting method \eqref{eq:5} converges with global order one in the $L^2(\Omega)$-norm.
\end{theorem}

\section{Numerical experiments}\label{sec:numerical}
In this section, the effectiveness of the stochastic Lie--Trotter splitting method and the stochastic Strang splitting method are validated through two numerical examples. To mitigate the potential impact of pseudo-random numbers on the numerical solutions, 1000 sample trajectories are employed in the computations. Since the exact solution is unknown, the solution obtained using the stochastic Strang splitting method with a step size of $h = 2^{-12}$ is used as the reference solution. The one order accuracy of the overall error between the approximate exact solution and the numerical solutions in the $L^2(\Omega)$-norm is then verified. Additionally, the phase error is computed to confirm that the numerical method preserves the discrete stochastic symplectic conservation law.
\subsection{Two-dimensional case}
In this example, the scenario involves a single prey species and a single predator species in the presence of one-dimensional noise. The parameters are  $\Gamma^{(1)} = \Gamma^{(2)} = 3$, $\eta^{(1)} = 5$, $\eta^{(2)} = 1$, $\Sigma^{(1)} = 1$ and $\Sigma^{(2)} = 0$ respectively. The initial values are $x(0)=1$ and $y(0)=7$. Under this setting, the stochastic LV model \eqref{eq:1} turns to be
\begin{equation*}
	\begin{cases}
		d x(t)=x(t)(-3 y(t)+1) d t, \\[2mm]
		d y(t)=y(t)[(3 x(t)-5) d t+dW(t)],\\[2mm]
		x(0)=1,~ y(0)=7.
	\end{cases}
\end{equation*}
Figure \ref{Fig1} (a) and (b) illustrates the trajectories of the numerical solutions $x(t)$ and $y(t)$ obtained using the two splitting methods. Compared with the reference solution over the time interval $t \in [0, 1]$.

\begin{figure}
	\begin{center}	\subfigure[]{
			\begin{minipage}[t]{0.45\linewidth}
				\includegraphics[width=1\textwidth]{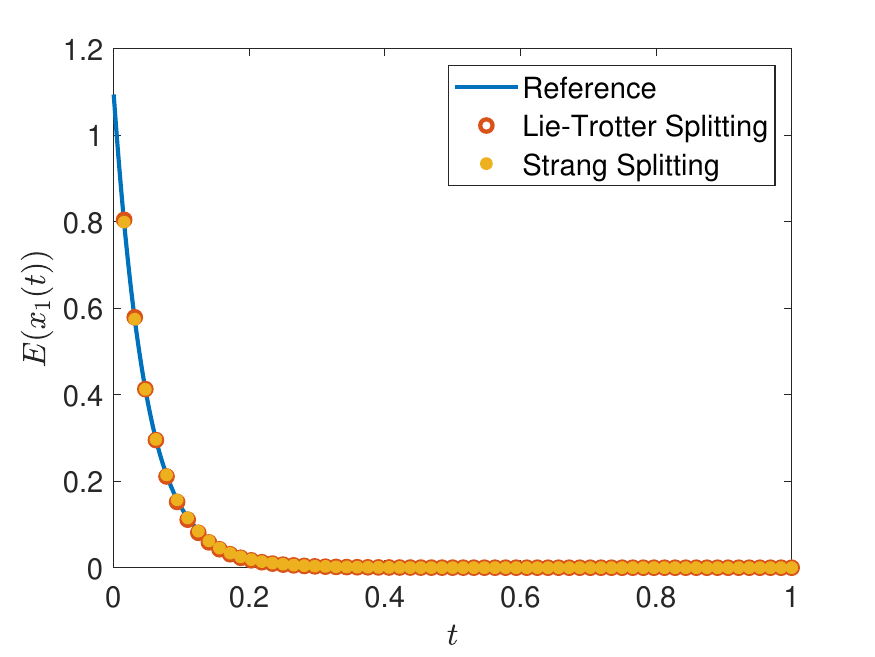}
			\end{minipage}
		}
		\subfigure[]{
			\begin{minipage}[t]{0.45\linewidth}
				\includegraphics[width=1\textwidth]{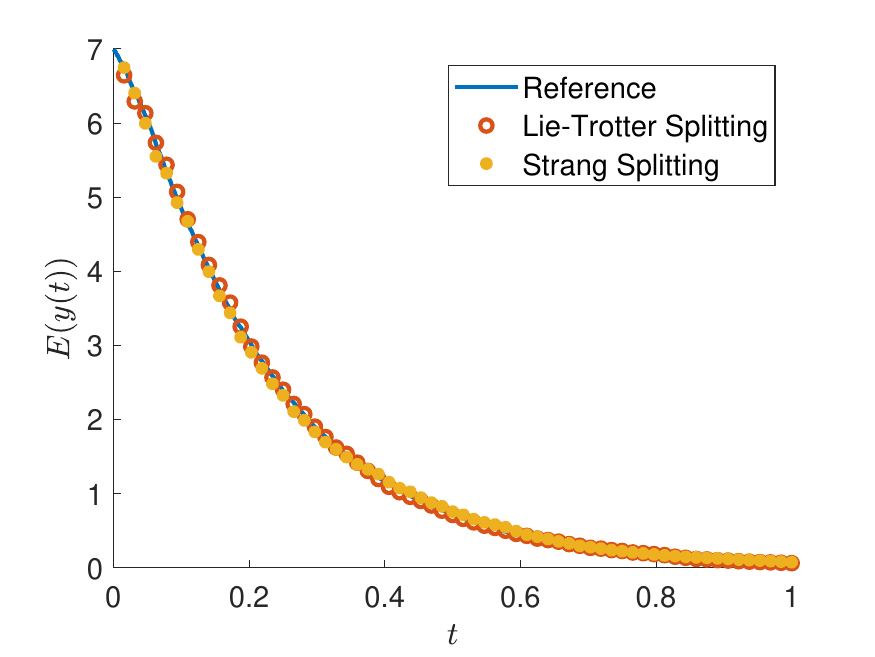}
			\end{minipage}
		}
	\end{center}
	\caption{
		(a) Schematic diagram of the trajectories for the numerical solution of $x(t)$ obtained by two stochastic splitting methods and the reference solution. (b) Schematic diagram of the trajectories for the numerical solution of $y(t)$ obtained by two stochastic splitting methods and the reference solution.
		\label{Fig1}}
\end{figure}

From Figure \ref{Fig1} (a) and (b), it can be observed that the numerical trajectories obtained using two stochastic splitting methods are closely distributed around the reference solution, indicating that the numerical solutions accurately approximate the reference solution.

To validate the conclusion from the previous section that the global error order of stochastic splitting methods are one, we set the final time $T = 1$ and consider six different step sizes $h = 2^{-i}$, $i = 4, 5, 6, 7, 8, 9$. The error is then plotted on a log-log scale to verify this result.
\begin{figure}
	\begin{center}
		\includegraphics[width=10cm]{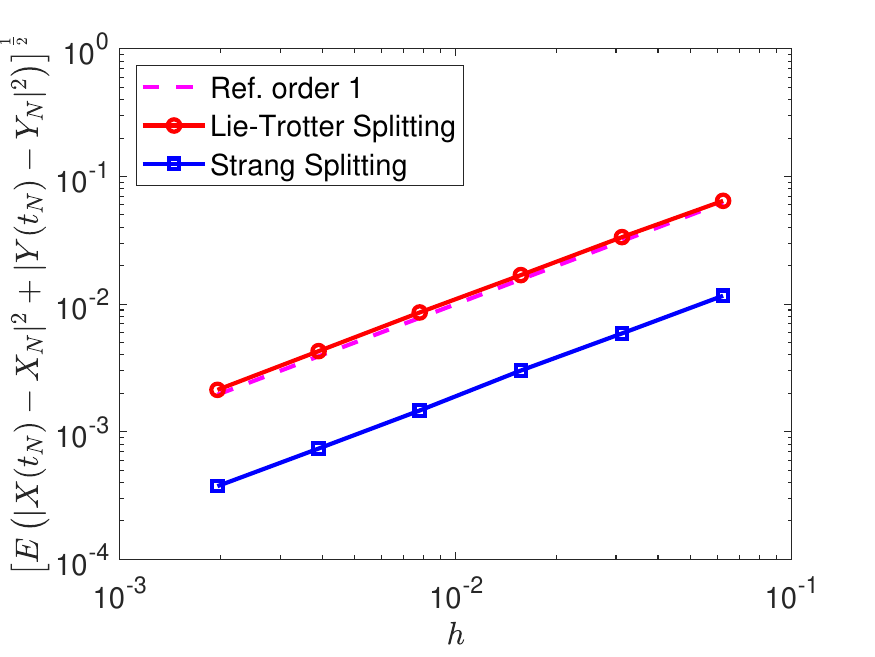}
	\end{center}
	\caption{
		The log-log error plot of two stochastic splitting methods for two-dimensional stochastic LV model in the $L^2(\Omega)$-norm.
		\label{Fig2}}
\end{figure}

From Figure \ref{Fig2}, it can be observed that the numerical results of both the stochastic Lie-Trotter splitting method and the stochastic Strang splitting method are parallel to the reference line with a convergence order of 1. This observation verifies the conclusion in Theorem \ref{thm3.9} that the global error order is one. 

The aforementioned numerical results effectively demonstrate the approximation accuracy of the two stochastic splitting methods to the reference solution. However, it fails to provide evidence regarding whether these methods can preserve the stochastic symplectic structure. To evaluate the performance of these two stochastic symplectic-preserving methods, we select $P_0^1=(1,7)$, $P_0^2=(2,1)$ and $P_0^3=(5,3)$. Within the time interval $[0, 10]$, the solutions $P_n^1=(x_n^1,y_n^1)$, $P_n^2=(x_n^2,y_n^2)$ and $P_n^3=(x_n^3,y_n^3)$ of the three points at time $t_n$ are obtained through the stochastic Lie--Trotter splitting method, the stochastic Strang splitting method and the EM method. Among them, Figure \ref{Fig3} (a) represents the phase space area $S_n$ of the triangle enclosed by these three points and
\begin{equation}{\notag}
	S_n =\frac{1}{2} 
	\begin{vmatrix}
		x_n^1 & y_n^1 & 1 \\
		x_n^2 & y_n^2 & 1 \\
		x_n^3 & y_n^3 & 1
	\end{vmatrix}.
\end{equation}
These results are compared with the triangle area $S_{R}$ of the reference solution within the time interval $[0, 10]$. The ordinate of Figure \ref{Fig3} (b) represents the difference in the phase space areas of the triangles enclosed by the three points of the three methods respectively and those of the reference solution, that is $\left|S_n-S_{R}\right|$.
\begin{figure}
	\begin{center}
		\subfigure[]{
			\begin{minipage}[t]{0.45\linewidth}
				\includegraphics[width=1\textwidth]{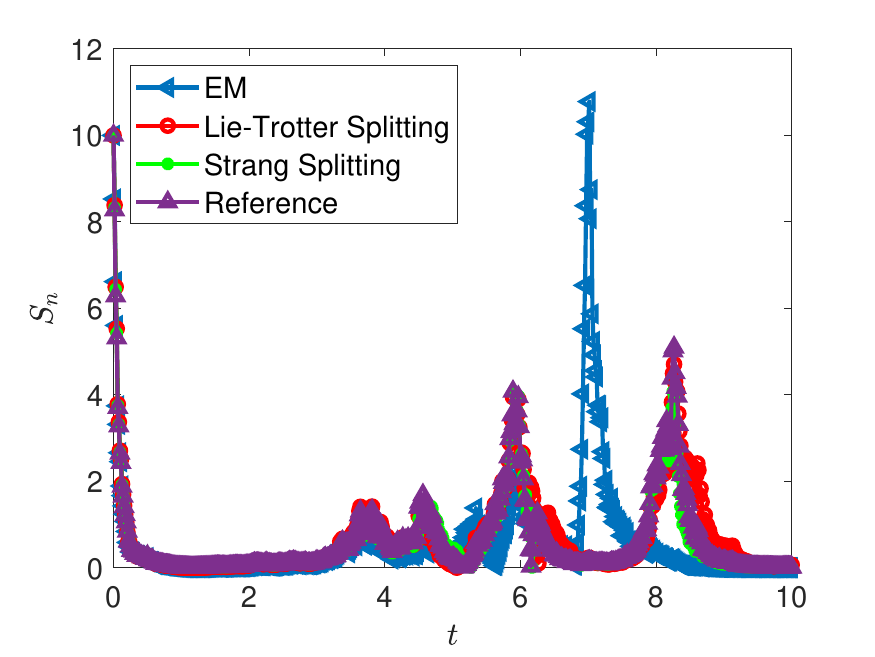}
			\end{minipage}
		}
		\subfigure[]{
			\begin{minipage}[t]{0.45\linewidth}
				\includegraphics[width=1\textwidth]{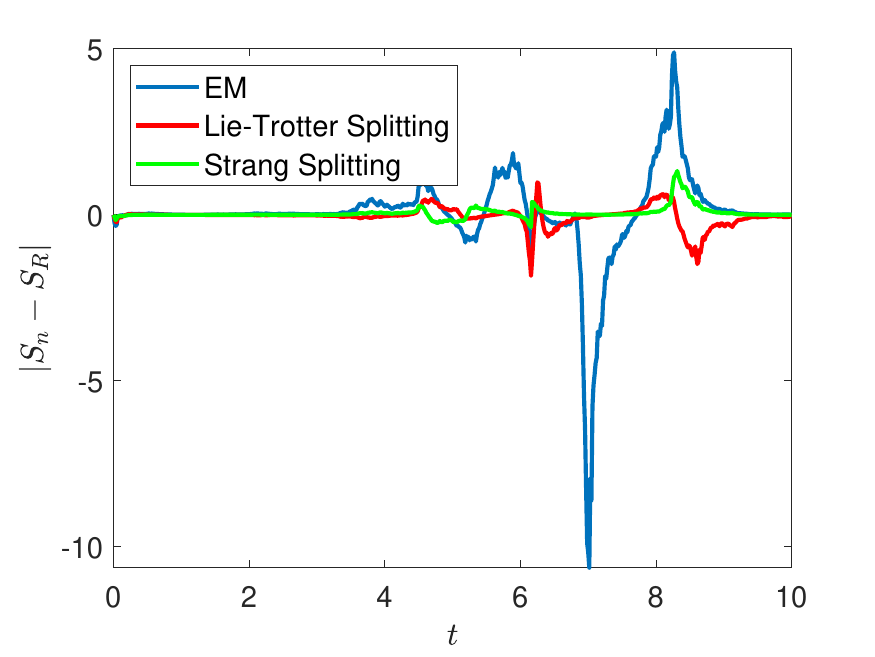}
			\end{minipage}
		}
	\end{center}
	\caption{
		A comparison chart of phase area between stochastic splitting methods and the EM method for exact solutions. (a) Comparison of the phase areas. (b) Comparison of the error of the phase space areas.
		\label{Fig3}}
\end{figure}
The results in Figure \ref{Fig3} (a) indicate that the phase areas of stochastic Lie--Trotter splitting method and stochastic Strang splitting method remain close to the reference solution's phase area over the time interval $[0, 10]$. In contrast, the EM numerical method exhibits significant deviation from the reference solution's phase area within the time interval $[6, 10]$. To more clearly illustrate the degree of deviation from the reference solution's phase area, the phase differences between each method and the reference solution are plotted. As shown in Figure \ref{Fig3} (b), the phase differences for the two stochastic symplectic-preserving splitting methods remain largely within a small region around $0$, while those for the non-symplectic-preserving EM method are notably larger. This observation highlights the superior performance of stochastic symplectic-preserving splitting methods.

\subsection{Four-dimensional case}
In this section, a scenario involving two prey species and two predator species in the presence of three-dimensional noise is considered. The selected parameters are as follows
$$\Gamma^{(1)}=\begin{bmatrix}{3}&{0}\\{0}&{5}\end{bmatrix},\quad \Gamma^{(2)}=\begin{bmatrix}{7}&{0}\\{0}&{4}\end{bmatrix},$$
$\eta^{(1)}=(1,4)^{\top}, \eta^{(2)}=(1,2)^{\top}, \Sigma^{(1)}=(0.4,0.5,0.6)$ and $\Sigma^{(2)}=(0,0,0)$. Under this setting, the stochastic LV model \eqref{eq:1} turns to be
\begin{equation*}\notag
	\begin{cases}
		d x_{1}(t)=x_{1}(t)\left(-7 y_{1}(t)+1\right) d t ,\\[2mm]\notag
		d x_{2}(t)=x_{2}(t)(-4 y_{2}(t)+2) d t , \\[2mm]\notag
		d y_{1}(t)=y_{1}(t)\left[(3 x_{1}(t)-1) d t+0.4 d W_{1}+0.5 d W_{2}+0.6 d W_{3}\right], \\[2mm]\notag
		d y_{2}(t)=y_{2}(t)\left[\left(5 x_{2}(t)-4\right) d t+0.4 d W_{1}+0.5 d W_{2}+0 . 6 d W_{3} \right],\\[2mm]\notag
		(x_1(0),x_2(0))=(1.1,5.2),~(y_1(0),y_2(0))=(3,7.1).
	\end{cases}
\end{equation*}
Figure \ref{FFig1} illustrates the trajectories of the numerical solutions $X(t) = (x_1(t), x_2(t))^{\top}$ and $Y(t) = (y_1(t), y_2(t))^{\top}$ obtained using two stochastic splitting methods, along with the reference solution, over the time interval $t \in [0, 1]$.
\begin{figure}
	\begin{center}
		\subfigure[]{
			\begin{minipage}[t]{0.45\linewidth}
				\includegraphics[width=1\textwidth]{./Figure1/x1-path.pdf}
			\end{minipage}
		}
		\subfigure[]{
			\begin{minipage}[t]{0.45\linewidth}
				\includegraphics[width=1\textwidth]{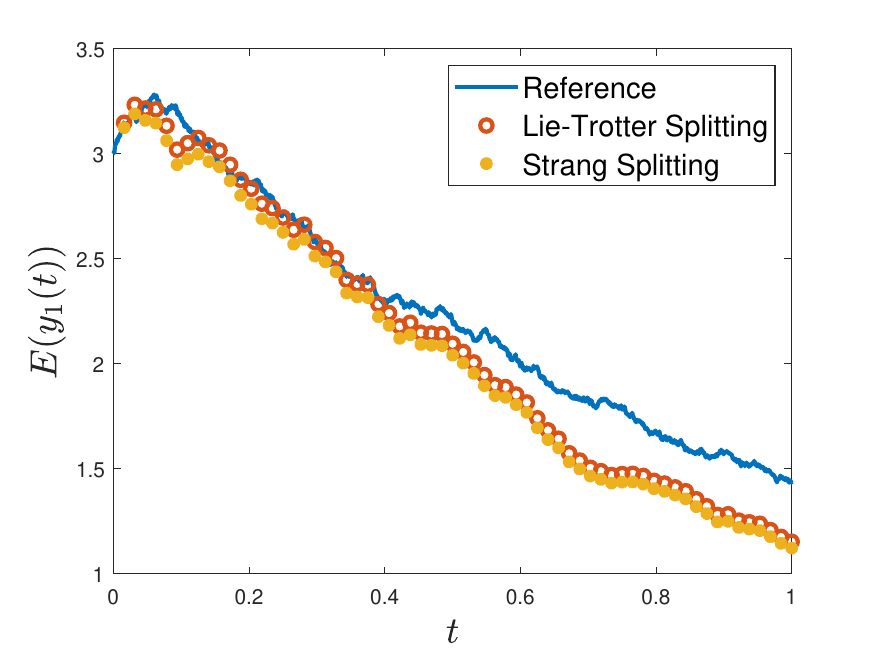}
			\end{minipage}
		}
		
		\subfigure[]{
			\begin{minipage}[t]{0.45\linewidth}
				\includegraphics[width=1\textwidth]{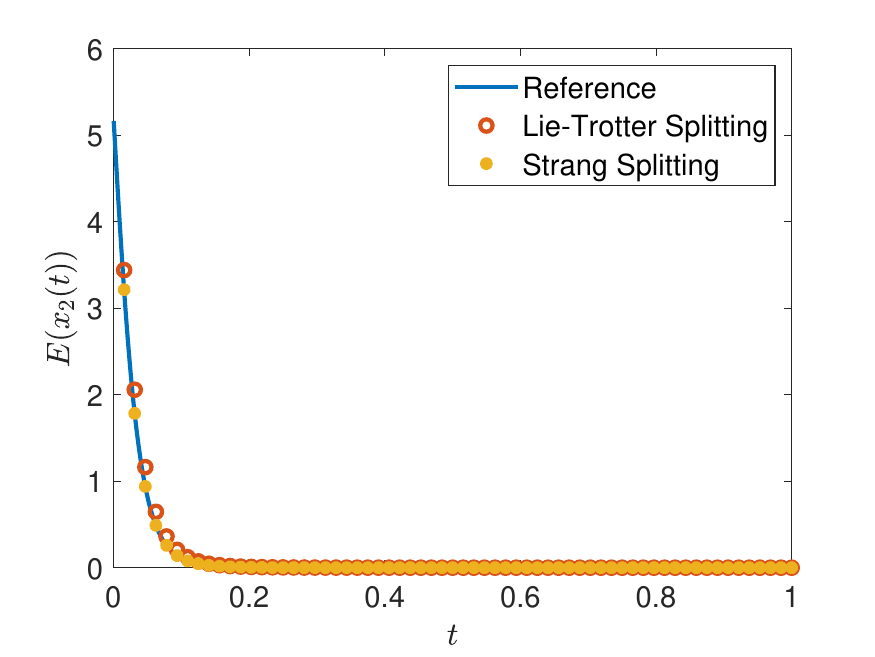}
			\end{minipage}
		}
		\subfigure[]{
			\begin{minipage}[t]{0.45\linewidth}
				\includegraphics[width=1\textwidth]{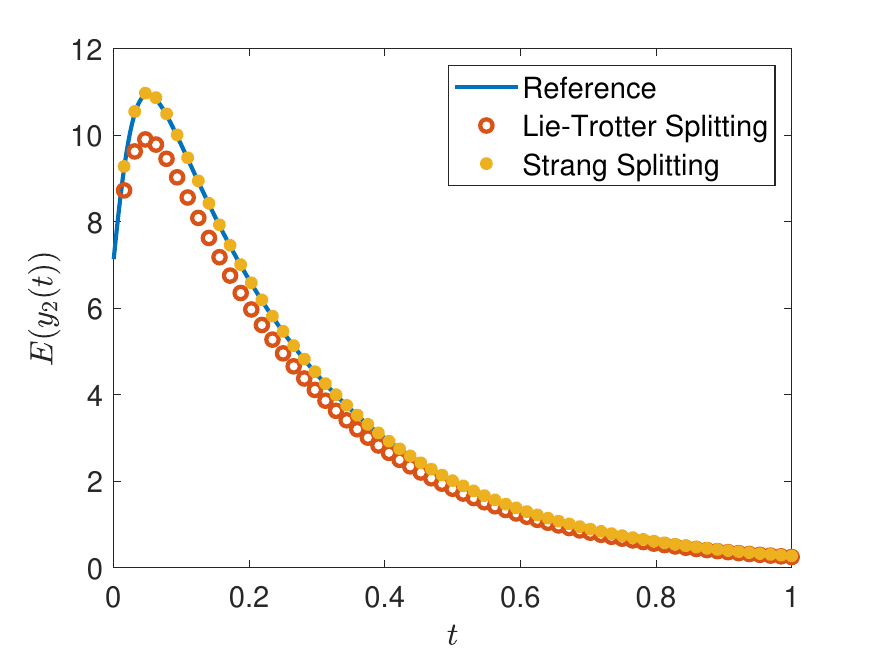}
			\end{minipage}
		}
	\end{center}
	\caption{
		(a) Schematic diagram of the trajectories for the numerical solution of $x_1(t)$ obtained by two stochastic splitting methods and the reference solution. (b) Schematic diagram of the trajectories for the numerical solution of $y_1(t)$ obtained by two stochastic splitting methods and the reference solution. (c) Schematic diagram of the trajectories for the numerical solution of $x_2(t)$ obtained by two stochastic splitting methods and the reference solution. (d) Schematic diagram of the trajectories for the numerical solution of $y_2(t)$ obtained by two stochastic splitting methods and the reference solution.
		\label{FFig1}}
\end{figure}
From Figure \ref{FFig1}, it can be observed that the numerical trajectories obtained using the two stochastic splitting methods are closely distributed around the reference solution, indicating that the numerical solutions accurately approximate the reference solution.
To further investigate the global error order of stochastic splitting methods, the time point $T = 1$ is considered and six different step sizes $h = 2^{-i}$ are selected for $i = 4, 5, 6, 7, 8, 9$. The errors are then plotted on a log-log scale to analyze the convergence behavior.
\begin{figure}
	\begin{center}
		\includegraphics[width=10cm]{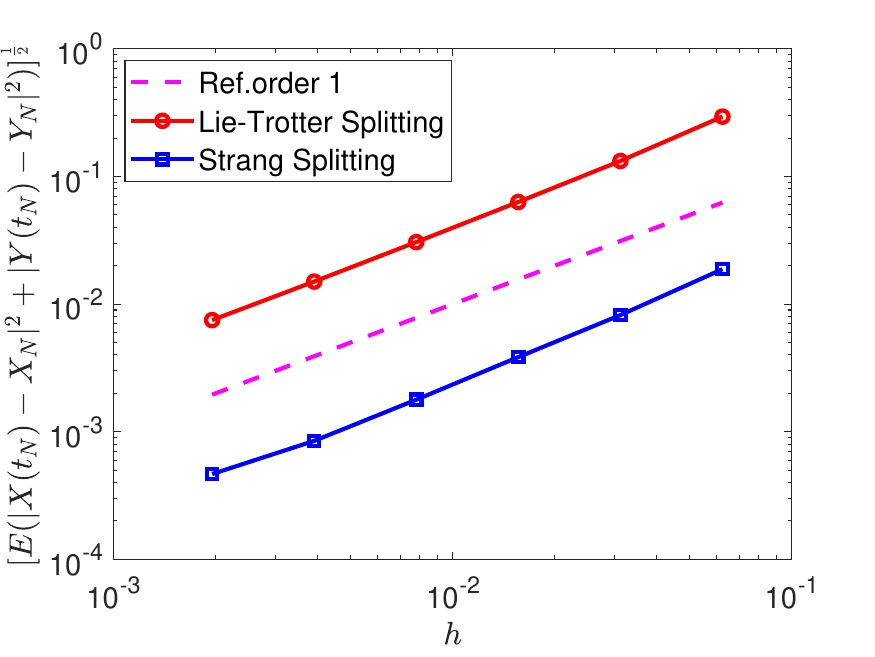}
	\end{center}
	\caption{
		The log-log error plot of two stochastic splitting methods for four-dimensional stochatic LV model in the $L^2(\Omega)$-norm.
		\label{FFig2}}
\end{figure}

The log-log plots of the two splitting methods in Figure \ref{FFig2} are parallel to the reference line with a convergence order of 1, thereby validating the global error order of 1 established in Section \ref{sec:positivity}.

\section{Conclusions}\label{sec:conclusion}
This paper investigates a class of numerical methods for the stochastic LV model. The findings reveal that under specific conditions, these methods exhibit a positivity-preserving symplectic structure and uniform boundedness. Furthermore, it is rigorously proven that the stochastic splitting methods converges with global order one in the $L^2(\Omega)$-norm. The conclusions are validated through comprehensive numerical examples in both two-dimensional and four-dimensional examples. However, stochastic splitting methods for solving the stochastic LV model have certain limitations. Notably, the conditions required to ensure one order convergence are stringent. Future research should explore whether more relaxed conditions can guarantee uniform boundedness and first-order convergence. Additionally, for more complex models, such as stochastic LV predator-prey models with noise components beyond white noise, it is crucial to investigate whether the system can maintain positivity and symplecticity. Developing numerical methods that preserve these properties in such scenarios would be of significant value.

\section*{Acknowledgments}
The authors would like to express their appreciation to the referees for their useful comments and the editors. Liying Zhang is supported by the National Natural Science Foundation of China (No. 11601514 and No. 11971458), the Fundamental Research Funds for the Central Universities (No. 2023ZKPYL02 and No. 2023JCCXLX01) and the Yueqi Youth Scholar Research Funds for the China University of Mining and Technology-Beijing (No. 2020YQLX03), 2025 Basic Sciences Initiative in Mathematics and Physics. Lihai Ji is supported by the National Natural Science Foundation of China (No. 12171047).


\end{document}